\documentclass[12pt,a4paper,leqno]{amsart}
\usepackage{amsmath, amsfonts,mathrsfs,amssymb,amscd,graphicx,setspace}
\usepackage[mathscr]{eucal}
\usepackage{mathtools}
\usepackage{tdclock}

\usepackage{hyperref,url}
\usepackage{enumitem}
\usepackage{color,soul} 

\newtheorem{theorem}{Theorem}

\newtheorem{lemma}{Lemma}

\allowdisplaybreaks
\dedicatory{}
\begin{document}
	\initclock
	\title[]{Convergence of multiple Dirichlet $L$-series}
	\author[Kohji Matsumoto]{Kohji Matsumoto}
	\address[Kohji Matsumoto]{Graduate School of Mathematics, Nagoya University, Furocho, Chikusa-ku, Nagoya 464-8602, Japan and Center for General Education, Aichi Institute of Technology, 1247 Yachigusa,
	Yakusa-cho, Toyota 470-0392, Japan}
	\email{ kohjimat@math.nagoya-u.ac.jp}
	\author[Dilip K. Sahoo]{Dilip K. Sahoo}
	\address[Dilip K. Sahoo]{Harish-Chandra Research Institute, Chhatnag Road, Jhunsi, 
Prayagraj(Allahabad) 211019, India}
	\email{mzfdilipmath@gmail.com}	
	\maketitle
	\begin{abstract}
		In this paper we study the convergence of multiple
		Dirichlet $L$-series defined by
			\begin{align*}
		\sum_{n_1=1}^{\infty}\sum_{n_2=1}^{\infty}\cdots	\sum_{n_r=1}^{\infty}\frac{\chi_1(n_1)\chi_2(n_2)\cdots\chi_r(n_r)}{n_1^{s_1}(n_1+n_2)^{s_2}\cdots (n_1+n_2+\cdots+n_r)^{s_r}}\,,
		\end{align*}where $ \chi_1,\chi_2,\ldots,\chi_r $ are non-principal Dirichlet characters.    We further consider a more general series involvong arithmetic functions
on the numerators. 
In particular we give an integral representation of the above series in the region of convergence by using Abel's summation formula. 
	\end{abstract}
\section{introduction}

	Multiple Dirichlet $L$-series is defined by
\begin{equation}
	\label{MDL-series}
	L(s_1,\ldots,s_r;\chi_1,\ldots,\chi_r)=\sum_{n_1=1}^{\infty}\sum_{n_2=1}^{\infty}\cdots	\sum_{n_r=1}^{\infty}\frac{\chi_1(n_1)\chi_2(n_2)\cdots\chi_r(n_r)}{n_1^{s_1}(n_1+n_2)^{s_2}\cdots (n_1+n_2+\cdots+n_r)^{s_r}}\,,
\end{equation}
where $r$ is a positive integer, $ \chi_1,\chi_2,\ldots,\chi_r $ are 
Dirichlet characters of modulo $ q_1,q_2,\ldots,q_r $ respectively and $ s_i\,(1\leq i\leq r) $ are complex variables. Throughout the article, we denote
$ \Re(s_i)=\sigma_i $. 
When all the characters $\chi_i$ are trivial (i.e., $\chi_i(n)=1$ for all $n$), then
\eqref{MDL-series} is the Euler-Zagier multiple series, 
which is known to be absolutely
convergent in the region 
$$ \mathcal{D}_0=
\{(s_1,s_2,\ldots,s_r)\in\mathbb{C}^r:\sigma_r+\sigma_{r-1}+\cdots+\sigma_{r-i}>i+1\,\,
{\rm for}\,\,0\leq i\leq r-1\} 
$$ 
(see \cite{Mat-Illinois}).
It is clear that in the same region, the series \eqref{MDL-series} for any 
$ \chi_1,\chi_2,\ldots,\chi_r $ is absolutely convergent.

Hereafter we assume that $ \chi_1,\chi_2,\ldots,\chi_r $ are non-principal.
For $ r=1 $, \eqref{MDL-series} is nothing but the Dirichlet $L$-series $ L(s,\chi) $. 
Therefore multiple Dirichlet $L$-series is  the multi-variable generalization of 
Dirichlet $L$-series. Matsumoto and Tanigawa \cite{Matsumoto} have studied the meromorphic continuation of multiple Dirichlet series and especially from \cite{Matsumoto}, it follows that $ 	L(s_1,\ldots,s_r;\chi_1,\ldots,\chi_r) $ is an entire function on $ \mathbb{C}^r $. The analytic properties of simillar type of multiple $L$-function defined by the series
\begin{align*}
\sum_{1\leq n_1<n_2<\cdots<n_r}\frac{\chi_1(n_1)\chi_2(n_2)\cdots\chi_r(n_r)}{n_1^{s_1}n_2^{s_2}\cdots n_r^{s_r}}	
\end{align*} is studied by Akiyama and Ishikawa \cite{Akiyama}.
\par
\vspace{2mm}
	It is well known that for non-principal character $ \chi $ modulo $ q $ , $ L(s,\chi)=\sum_{n=1}^{\infty}\frac{\chi(n)}{n^{s}} $ converges for $ \Re(s)>0 $. Because by applying the Abel summation formula, we get that
\begin{align}
	\label{exp2}
	\sum_{n\leq x}\frac{\chi(n)}{n^s}=\left(\sum_{n\leq x}\chi(n)\right)\frac{1}{x^s}+s\int_{1}^x\left(\sum_{n\leq t}\chi(n)\right)\frac{1}{t^{s+1}}dt\,.	
\end{align}Here $ \displaystyle\lim_{x\rightarrow\infty}\left(\sum_{n\leq x}\chi(n)\right)\frac{1}{x^s}=0$ for  $ \Re(s)>0 $ because $ \left|\displaystyle\sum_{n\leq x}\chi(n)\right|\leq q $, and also 	$ \displaystyle\int_{1}^\infty\left(\sum_{n\leq t}\chi(n)\right)\frac{1}{t^{s+1}}dt $ exists for  $ \Re(s)>0 $. Therefore \eqref{exp2} implies that
\begin{align}
	\label{exp3}
	\sum_{n=1}^{\infty}\frac{\chi(n)}{n^{s}}=s\int_{1}^\infty\left(\sum_{n\leq t}\chi(n)\right)\frac{1}{t^{s+1}}dt\,\,\text{for}\,\,\Re(s)>0.	
\end{align}
More generally, for any non-negative integer $n_0$, we obtain
\begin{align}
	\label{exp3-2}
	\sum_{n=1}^{\infty}\frac{\chi(n)}{(n_0+n)^{s}}=s\int_{1}^\infty\left(\sum_{n\leq t}\chi(n)\right)\frac{1}{(n_0+t)^{s+1}}dt\,\,\text{for}\,\,\Re(s)>0.	
\end{align}

 Our first aim of this article is to find the region of (not necessarily absolute, but conditional) convergence of the multiple Dirichlet $L$-series \eqref{MDL-series}. 
Here we understand that the order of summation in \eqref{MDL-series} is, as indicated, 
first take the sum with respect to $n_r$, then to $n_{r-1}$, $\ldots$ and finally to
$n_1$.
\begin{theorem}\label{conv_result}
The multiple series \eqref{MDL-series} is (conditionally) convergent in the region
$$\mathcal{D}=\{\{(s_1,s_2,\ldots,s_r)\in\mathbb{C}^r:\sigma_r+\sigma_{r-1}+\cdots+\sigma_{r-i}>0\,\,{\rm for}\,\,0\leq i\leq r-1\}.
$$
\end{theorem}

The numerator of \eqref{MDL-series} consists of Dirichlet characters, but it is possible to
extend the result to the case of more general numerators.

Let $ a_j(m)$ $(1\leq j\leq r)$ be arithmetic functions, and define
$$
\Phi(s_1,\ldots,s_r;a_1,\ldots,a_r)=
\sum_{n_1=1}^{\infty}\sum_{n_2=1}^{\infty}\cdots	\sum_{n_r=1}^{\infty}\frac{a_1(n_1)a_2(n_2)\cdots a_r(n_r)}{n_1^{s_1}(n_1+n_2)^{s_2}\cdots (n_1+n_2+\cdots+n_r)^{s_r}}.
$$
This form of general multiple series was already introduced in \cite{Matsumoto}, and
meromorphic continuation etc. were discussed.
The second-named author \cite{Sahoo} studied the domain of absolute convergence of
$\Phi(s_1,\ldots,s_r;a_1,\ldots,a_r)$.

Here we state the following generalization of Theorem \ref{conv_result} for this multiple
series $\Phi(s_1,\ldots,s_r;a_1,\ldots,a_r)$.
The proof is
quite the same as that of Theorem \ref{conv_result}, so in the present paper we only
describe the proof of Theorem \ref{conv_result}.

\begin{theorem} \label{generalization}
      Assume that for each $j$,
there exists a positive real number $ \alpha_j $ for which the condition 
$ \left|\sum_{m\leq t} a_j(m)\right|\leq\alpha_j $ 
holds for every $ t\geq 1 $. Then $\Phi(s_1,\ldots,s_r;a_1,\ldots,a_r)$ is
(conditionally) convergent in the region $\mathcal{D}$.    
%
	\end{theorem}
 
Now we mention some relevant history and the motivation of the present paper.

An important different case, where numerators are twised by certain exponential factors, 
is that of multiple polylogarithms. 
In the one-variable case, such series were already studied by several papers of
T. Shintani and of P. Cassou-Nogu{\`e}s around 1980.
The multi-variable situation was treated later by Imai \cite{Ima81}
and Hida \cite{Hid93} in connection with algebraic number theory, 
by Goncharov \cite{Gon95} \cite{Gon-prep} \cite{Gon01} from the viewpoint of arithmetic
geometry, and from more analytic aspects by de Crisenoy \cite{deC06}, 
de Crisenoy and Essouabri \cite{deCEss08}, and by Essouabri and Matsumoto \cite{EssMat19}.
The multi-variable series with more general form of numerators was introduced by
Lichtin (see \cite{Lic93} \cite{Lic94}), and he applied his results to a certain
lattice point problem.   Ishikawa \cite{Ish02} applied the result in \cite{Akiyama} 
to the evaluation of
multiple character sums.    More generally, the application of the theory of multiple series
to the study of asymptotic properties of arithmetical functions has been considered by a lot
of articles, such as
de la Bret{\`e}che \cite{dlB01}, Bhowmik et al. \cite{BEL07}, Essouabri \cite{Ess12},
\cite{Ess20}, and Essouabri et al. \cite{ESZT22}.

In order to develop such investigations with various applications, it is essentially important 
to understand the complex-analytic behavior of relevant multiple series, such as the region of 
convergence, meromorphic continuation and the location of singularities.   We believe that
the result in the present article is the first step in this direction. 

Here we explain the main idea of the present article.
The strategy of the proof of Theorem \ref{conv_result}
is to express a slightly generalized multiple series 
in terms of an integral which is the multi-variable generalization of \eqref{exp3-2}. We state this result in the following theorem.   
\begin{theorem}\label{main result}
In the region $\mathcal{D}$, we have 	
 	\begin{align*}
 		&\sum_{n_1=1}^{\infty}\sum_{n_2=1}^{\infty}\cdots	\sum_{n_r=1}^{\infty}\frac{\chi_1(n_1)\chi_2(n_2)\cdots\chi_r(n_r)}{(n_0+n_1)^{s_1}(n_0+n_1+n_2)^{s_2}\cdots (n_0+n_1+n_2+\cdots+n_r)^{s_r}}\\
 		&=\int_{1}^\infty\int_{1}^\infty\cdots\int_{1}^\infty\prod_{i=1}^r\left(\sum_{n_i\leq t_i}\chi_i(n_i)\right)\\
 		&\times\left(\sum_{\substack{k_1+k_2+\cdots+k_r=r\\k_1+\cdots+k_i\leq i\,,\,1\leq i\leq r\\0\leq k_i\leq i\,,\,1\leq i\leq r}}{1\choose k_1}\prod_{i=2}^r{i-k_1-\cdots-k_{i-1}\choose k_i}\prod_{i=1}^r\frac{(s_i)_{k_i}}{\left(n_0+\sum_{j=1}^i t_j\right)^{s_i+k_i}}\right)dt_rdt_{r-1}\cdots dt_1\,.
 	\end{align*}
 \end{theorem}

The main tool to prove Theorem \eqref{main result} is the Abel summation formula. 
Therefore our method is of elementary nature, but technically it is not so simple;
the key technical point is embodied in Lemma \ref{lemma1}, stated in Section \ref{sec3}.

It is to be stressed that, when $n_0=0$, Theorem \ref{main result} gives an integral expression of
\eqref{MDL-series} in the region $\mathcal{D}$.   For instance, in the case $r=2$, 
the above theorem implies 
\begin{align}\label{case2}
		&\sum_{n_1=1}^{\infty}\sum_{n_2=1}^{\infty}\frac{\chi_1(n_1)\chi_2(n_2)}{n_1^{s_1}(n_1+n_2)^{s_2}}
		\\
		&=\int_{1}^{\infty}\int_{1}^{\infty}\left(\sum_{n_1\leq t_1}\chi_1(n_1)\right)\left(\sum_{n_2\leq t_2}\chi_2(n_2)\right)\left(\frac{s_1s_2}{t_1^{s_1+1}(t_1+t_2)^{s_2+1}}+\frac{s_2(s_2+1)}{t_1^{s_1}(t_1+t_2)^{s_2+2}}\right)dt_2dt_1	\notag
\end{align} 
for $\sigma_2>0$, $\sigma_2+\sigma_1>0$.

Similar expression can be shown for the generalized series 
$\Phi(s_1,\ldots,s_r;a_1,\ldots,a_r)$:

\begin{theorem}
In the region $\mathcal{D}$, the integral expression
	\begin{align*}
		&\Phi(s_1,\ldots,s_r;a_1,\ldots,a_r)\\
		&=\int_{1}^\infty\int_{1}^\infty\cdots\int_{1}^\infty\prod_{i=1}^r\left(\sum_{n_i\leq t_i}a_i(n_i)\right)\\
		&\times\left(\sum_{\substack{k_1+k_2+\cdots+k_r=r\\k_1+\cdots+k_i\leq i\,,\,1\leq i\leq r\\0\leq k_i\leq i\,,\,1\leq i\leq r}}{1\choose k_1}\prod_{i=2}^r{i-k_1-\cdots-k_{i-1}\choose k_i}\prod_{i=1}^r\frac{(s_i)_{k_i}}{\left(\sum_{j=1}^i t_j\right)^{s_i+k_i}}\right)dt_rdt_{r-1}\cdots dt_1\,
	\end{align*}
holds.
\end{theorem}



\section{The case $r=2$}\label{sec2}

In this section we prove the theorems in the case $r=2$.    The basic structure of
our method can already be seen in this special case.

		For $ \sigma_2>0 $, it follows from \eqref{exp3-2} that
	\begin{align}
		\label{exp4}
		&\sum_{n_1\leq x}\sum_{n_2=1}^{\infty}\frac{\chi_1(n_1)\chi_2(n_2)}{(n_0+n_1)^{s_1}(n_0+n_1+n_2)^{s_2}}\\
		&=\sum_{n_1\leq x}\frac{\chi_1(n_1)}{(n_0+n_1)^{s_1}}\left(\int_{1}^{\infty}\left(\sum_{n_2\leq t_2}\chi_2(n_2)\right)\frac{s_2}{(n_0+n_1+t_2)^{s_2+1}}dt_2\right)\nonumber\\
		&=s_2\int_{1}^{\infty}\left(\sum_{n_2\leq t_2}\chi_2(n_2)\right)\left(\sum_{n_1\leq x}\frac{\chi_1(n_1)}{(n_0+n_1)^{s_1}(n_0+n_1+t_2)^{s_2+1}}\right)dt_2\,.\nonumber
	\end{align}
	By applying the Abel summation formula to the 
last inner sum,	we get from \eqref{exp4} that
		\begin{align}\label{exp5}
	&\sum_{n_1\leq x}\sum_{n_2=1}^{\infty}\frac{\chi_1(n_1)\chi_2(n_2)}{(n_0+n_1)^{s_1}(n_0+n_1+n_2)^{s_2}}\\
	&=s_2\int_{1}^{\infty}\left(\sum_{n_2\leq t_2}\chi_2(n_2)\right)\Biggl\{\frac{\sum_{n_1\leq x}\chi_1(n_1)}{(n_0+x)^{s_1}(n_0+x+t_2)^{s_2+1}}
			+\int_{1}^{x}\left(\sum_{n_1\leq t_1}\chi_1(n_1)\right)\notag\\
	&\times\left(\frac{s_1}{(n_0+t_1)^{s_1+1}(n_0+t_1+t_2)^{s_2+1}}+\frac{s_2+1}{(n_0+t_1)^{s_1}(n_0+t_1+t_2)^{s_2+2}}\right)dt_1\Biggr\}dt_2\notag\\
				=&\frac{s_2}{(n_0+x)^{s_1}}\left(\sum_{n_1\leq x}\chi_1(n_1)\right)\int_{1}^{\infty}\left(\sum_{n_2\leq t_2}\chi_2(n_2)\right)\frac{1}{(n_0+x+t_2)^{s_2+1}}dt_2\notag\\
			&\hspace{6mm}+s_2\int_{1}^{\infty}\left(\sum_{n_2\leq t_2}\chi_2(n_2)\right)\left(\int_{1}^{x}\left(\sum_{n_1\leq t_1}\chi_1(n_1)\right)\frac{s_1}{(n_0+t_1)^{s_1+1}(n_0+t_1+t_2)^{s_2+1}}dt_1\right)dt_2\notag\\
			&\hspace{6mm}+s_2\int_{1}^{\infty}\left(\sum_{n_2\leq t_2}\chi_2(n_2)\right)\left(\int_{1}^{x}\left(\sum_{n_1\leq t_1}\chi_1(n_1)\right)\frac{s_2+1}{(n_0+t_1)^{s_1}(n_0+t_1+t_2)^{s_2+2}}dt_1\right)dt_2 \notag\\
			&=J_1+J_2+J_3,\notag
		\end{align}
say.	Now we will discuss the existence of the integrals on the right-hand side of \eqref{exp5}  when $ x $ tends to $ \infty $. 
	Here for $  \sigma_2>0\,,\,\sigma_2+\sigma_1>0 $, we have 
	\begin{align*}
	\lim_{x\to\infty}J_1=0
	\end{align*} 
as
	\begin{align*}
	|J_1|
		&\leq\frac{\left|s_2\right|q_1q_2}{(n_0+x)^{\sigma_1}}\int_{1}^{\infty}\frac{1}{(n_0+x+t_2)^{\sigma_2+1}}dt\\
		&=\frac{\left|s_2\right|q_1q_2}{\sigma_2(n_0+x)^{\sigma_1}(n_0+x+1)^{\sigma_2}}\leq\frac{\left|s_2\right|q_1q_2}{\sigma_2(n_0+x)^{\sigma_1+\sigma_2}}\,.
	\end{align*}
Next we show that $\lim_{x\to\infty}J_2$ exists
for  $ \sigma_2>0\,,\,\sigma_2+\sigma_1>0 $ 
as
	\begin{align*}
	&|J_2|\\
	&\leq\left|s_1s_2\right|q_1q_2\int_{1}^{\infty}\int_{1}^{x}\frac{1}{(n_0+t_1)^{\sigma_1+1}(n_0+t_1+t_2)^{\sigma_2+1}}dt_1dt_2\\
	&=\left|s_1s_2\right|q_1q_2\int_{1}^{x}\frac{1}{(n_0+t_1)^{\sigma_1+1}}\left(\int_{1}^{\infty}\frac{1}{(n_0+t_1+t_2)^{\sigma_2+1}}dt_2\right)dt_1\\
	&=\frac{\left|s_1s_2\right|q_1q_2}{\sigma_2}\int_{1}^{x}\frac{1}{(n_0+t_1)^{\sigma_1+1}(n_0+t_1+1)^{\sigma_2}}dt_1\\
	&\leq\frac{\left|s_1s_2\right|q_1q_2}{\sigma_2}\int_{1}^{x}\frac{1}{(n_0+t_1)^{\sigma_1+\sigma_2+1}}dt_1=\frac{\left|s_1s_2\right|q_1q_2}{\sigma_2(\sigma_2+\sigma_1)}
	\left(\frac{1}{(n_0+1)^{\sigma_2+\sigma_1}}-\frac{1}{(n_0+x)^{\sigma_2+\sigma_1}}\right).	
\end{align*}
	Simillarly we can check that $\lim_{x\to\infty}J_3$ 
exists for  $ \sigma_2>0\,,\,\sigma_2+\sigma_1>0 $.
for  $ \sigma_2>0\,,\,\sigma_2+\sigma_1>0 $.

This implies that, when $x\to\infty$, the limit of the left-hand side of \eqref{exp4}
exists, and is equal to 
\begin{align*}
&\lim_{x\to\infty}(J_2+J_3)\\
&=\int_{1}^{\infty}\int_{1}^{\infty}\left(\sum_{n_1\leq t_1}\chi_1(n_1)\right)\left(\sum_{n_2\leq t_2}\chi_2(n_2)\right)\left(\frac{s_1s_2}{(n_0+t_1)^{s_1+1}(n_0+t_1+t_2)^{s_2+1}}\right.\\
&\qquad\left.+\frac{s_2(s_2+1)}{(n_0+t_1)^{s_1}(n_0+t_1+t_2)^{s_2+2}}\right)dt_2dt_1.
\end{align*}
Therefore we now establish our main results Theorem \ref{conv_result} and
Theorem \ref{main result} in the case $r=2$.

\section{A preliminary lemma}\label{sec3}

In order to prove the theorems in general case, we first prepare an important lemma.
Let $(s)_k=s(s+1)\cdots(s+k-1)$ for any positive integer $k$, and $(s)_0=1$.

\begin{lemma}
	\label{lemma1}
	For integer $ r\geq 3 $, we have
	\begin{align}\label{lem1}
		&-\sum_{\substack{k_2+k_3+\cdots+k_{r}=r-1\\k_2+\cdots+k_i\leq i-1\,,\,2\leq i\leq r\\0\leq k_i\leq i-1\,,\,2\leq i\leq r}}{1\choose k_2}\prod_{i=3}^r{i-1-k_2-\cdots-k_{i-1}\choose k_{i}}\\
&\qquad\times		\frac{\partial}{\partial t_1}\left(\frac{1}{(n_0+t_1)^{s_1}}\prod_{i=2}^r\frac{(s_i)_{k_i}}{\left(n_0+\sum_{j=1}^i t_j\right)^{s_i+k_i}}\right)\notag\\
		&=\sum_{\substack{m_1+m_2+\cdots+m_r=r\\m_1+\cdots+m_i\leq i\,,\,1\leq i\leq r\\0\leq m_i\leq i\,,\,1\leq i\leq r}}{1\choose m_1}\prod_{i=2}^r{i-m_1-\cdots-m_{i-1}\choose m_i}\prod_{i=1}^r\frac{(s_i)_{m_i}}{\left(n_0+\sum_{j=1}^i t_j\right)^{s_i+m_i}}\,.	\notag
	\end{align}
\end{lemma}

\begin{proof}
For brevity we write 
\begin{align*}
P_i(k_i)=\frac{(s_i)_{k_i}}{\left(n_0+\sum_{j=1}^i t_j\right)^{s_i+k_i}}.
\end{align*}
Since
\begin{align*}
&\frac{\partial}{\partial t_1}\left(\frac{1}{(n_0+t_1)^{s_1}}\prod_{i=2}^r
		P_i(k_i)\right)\\
&\qquad= -\frac{s_1}{(n_0+t_1)^{s_1+1}}\prod_{i=2}^r P_i(k_i)
		-\frac{1}{(n_0+t_1)^{s_1}}\sum_{h=2}^r P_h(k_h+1) 
		\prod_{\substack{i=2\\i\neq h}}^r P_i(k_i)
\end{align*}
we find that the left-hand side of \eqref{lem1} is
	\begin{align}\label{LHS}
		&-\sum_{\substack{k_2+k_3+\cdots+k_{r}=r-1\\k_2+\cdots+k_i\leq i-1\,,\,2\leq i\leq r\\0\leq k_i\leq i-1\,,\,2\leq i\leq r}}{1\choose k_2}\prod_{i=3}^r{i-1-k_2-\cdots-k_{i-1}\choose k_{i}}
		\frac{\partial}{\partial t_1}\left(\frac{1}{(n_0+t_1)^{s_1}}\prod_{i=2}^r
		P_i(k_i)
		\right)
		\\
&=\frac{s_1}{(n_0+t_1)^{s_1+1}}S+\frac{1}{(n_0+t_1)^{s_1}}\sum_{h=2}^r S_h,\notag
		\end{align}
where
$$
S=\sum_{\substack{k_2+k_3+\cdots+k_{r}=r-1\\k_2+\cdots+k_i\leq i-1\,,\,2\leq i\leq r\\0\leq k_i\leq i-1\,,\,2\leq i\leq r}}{1\choose k_2}\prod_{i=3}^r{i-1-k_2-\cdots-k_{i-1}\choose k_{i}}
\prod_{i=2}^r P_i(k_i)
$$
and
$$
S_h=\sum_{\substack{k_2+k_3+\cdots+k_{r}=r-1\\k_2+\cdots+k_i\leq i-1\,,\,2\leq i\leq r\\0\leq k_i\leq i-1\,,\,2\leq i\leq r}}{1\choose k_2}\prod_{i=3}^r{i-1-k_2-\cdots-k_{i-1}\choose k_{i}}
P_h(k_h+1)\prod_{\substack{i=2\\i\neq h}}^r P_i(k_i)
$$	
for $2\leq h\leq r$.	

On the other hand, on the right-hand side of \eqref{lem1}, the only possible values of
$m_1$ are $ m_1=0,1 $. 
Denote the parts 
 of the right-hand side corresponding to $m_1=0$ and $m_1=1$ by $A_0$ and $A_1$, respectively.
It is easy to see that the part $A_1$ is
	\begin{align*}
		A_1&=\sum_{\substack{m_2+\cdots+m_r=r-1\\m_2+\cdots+m_i\leq i-1\,,\,2\leq i\leq r\\0\leq m_i\leq i-1\,,\,2\leq i\leq r}}{1\choose m_2}\prod_{i=3}^r{i-1-m_2-\cdots-m_{i-1}\choose m_i}\frac{s_1}{(n_0+t_1)^{s_1+1}}\prod_{i=2}^r P_i(m_i)\\
		& = \frac{s_1}{(n_0+t_1)^{s_1+1}}S.
	\end{align*}
From this and \eqref{LHS} we see that our remaining task is to show
\begin{align}\label{lem-1-aim}
A_0=\frac{1}{(n_0+t_1)^{s_1}}\sum_{h=2}^r S_h.
\end{align}
Applying the simple combinatorial identity
$\displaystyle{\binom{2}{m_2}=\binom{1}{m_2-1}+\binom{1}{m_2}}$
(where, and in what follows, we understand $\binom{n}{k}=0$ if $k<0$ or if $n<k$),
we have
\begin{align*}
A_0&=\sum_{\substack{m_2+\cdots+m_r=r\\m_2+\cdots+m_i\leq i\,,\,2\leq i\leq r\\0\leq m_i\leq i\,,\,2\leq i\leq r}}{2\choose m_2}\prod_{i=3}^r{i-m_2-\cdots-m_{i-1}\choose m_i}\frac{1}{(n_0+t_1)^{s_1}}\prod_{i=2}^r P_i(m_i)\\
&=\frac{1}{(n_0+t_1)^{s_1}}(B_2+C_2),
\end{align*}
where
$$
B_2=\sum_{\substack{m_2+\cdots+m_r=r\\m_2+\cdots+m_i\leq i\,,\,2\leq i\leq r\\
1\leq m_2\leq 2\\0\leq m_i\leq i\,,\,3\leq i\leq r}}{1\choose m_2-1}\prod_{i=3}^r{i-m_2-\cdots-m_{i-1}\choose m_i}\prod_{i=2}^r P_i(m_i)
$$
and
$$
C_2=\sum_{\substack{m_2+\cdots+m_r=r\\m_2+\cdots+m_i\leq i\,,\,3\leq i\leq r\\
0\leq m_2\leq 1\\0\leq m_i\leq i\,,\,3\leq i\leq r}}{1\choose m_2}{3-m_2\choose m_3}\prod_{i=4}^r{i-m_2-\cdots-m_{i-1}\choose m_i}\prod_{i=2}^r P_i(m_i).
$$
Putting $k_2=m_2-1$ and $k_i=m_i$ ($3\leq i\leq r$), we find that $B_2=S_2$, hence
$A_0=(n_0+t_1)^{-s_1}(S_2+C_2)$.
Next, according to
$\displaystyle{{3-m_2\choose m_3}={2-m_2\choose m_3-1}+{2-m_2\choose m_3}}$, 
we divide $C_2$ as above into two parts: $C_2=B_3+C_3$, where
$$
B_3=\sum_{\substack{m_2+\cdots+m_r=r\\m_2+\cdots+m_i\leq i\,,\,2\leq i\leq r\\
0\leq m_2\leq 1, 1\leq m_3\leq 3\\0\leq m_i\leq i\,,\,4\leq i\leq r}}{1\choose m_2}{2-m_2\choose m_3-1}
\prod_{i=4}^r{i-m_2-\cdots-m_{i-1}\choose m_i}\prod_{i=2}^r P_i(m_i)
$$
and
\begin{align*}
C_3&=\sum_{\substack{m_2+\cdots+m_r=r\\m_2+m_3\leq 2\\m_2+\cdots+m_i\leq i\,,\,4\leq i\leq r
\\0\leq m_2\leq 1,0\leq m_3\leq 2\\0\leq m_i\leq i\,,\,4\leq i\leq r}}{1\choose m_2}{2-m_2\choose m_3}{4-m_2-m_3\choose m_4}\\
&\times\prod_{i=5}^r{i-m_2-\cdots-m_{i-1}\choose m_i}\prod_{i=2}^r P_i(m_i).
\end{align*}

Then, putting $k_3=m_3-1$ and $k_i=m_i$ ($i\neq 3$), we find that $B_3=S_3$, so
$A_0=(n_0+t_1)^{-s_1}(S_2+S_3+C_3)$.
Next using
$$
{4-m_2-m_3\choose m_4}={3-m_2-m_3\choose m_4-1}+{3-m_2-m_3\choose m_4}
$$
we divide $C_3$ in a similar way, and so on.
Repeating this procedure, we arrive at
$A_0=(n_0+t_1)^{-s_1}(S_2+S_3+\cdots+S_{r-1}+C_{r-1})$, where
\begin{align*}
C_{r-1}&=\sum_{\substack{m_2+\cdots+m_r=r\\m_2+\cdots+m_i\leq i-1\,,\,2\leq i\leq r-1\\0\leq m_i\leq i-1\,,\,2\leq i\leq r-1\\0\leq m_r\leq r}}{1\choose m_2}{2-m_2\choose m_3}\cdots
{(r-2)-m_2-\cdots-m_{r-2}\choose m_{r-1}}\\
&\times{r-m_2-\cdots-m_{r-1}\choose m_r}\prod_{i=2}^r P_i(m_i).
\end{align*}
On $C_{r-1}$, the conditions $m_2+\cdots+m_r=r$ and $m_2+\cdots+m_{r-1}\leq r-2$
yield $m_r\geq 2$.   Therefore we may put $k_r=m_r-1$ and $k_i=m_i$ ($i\neq r$).
Since
$$
{r-m_2-\cdots-m_{r-1}\choose m_r}=1={r-1-k_2-\cdots-k_{r-1}\choose k_r},
$$
we have $C_{r-1}=S_r$, and hence \eqref{lem-1-aim}.
The proof of the lemma is complete.
	
	\end{proof}
\section{Proof of Theorems \ref{conv_result} and \ref{main result}}
\begin{proof}
		We prove the theorems by induction on $ r $.
	\vspace{2mm}\\
	We already checked that the theorems are true for $ r=1,2 $ cases.
	\vspace{2mm}\\
	Assume the theorems are true for $ r-1 $, $ r\geq3 $. Now we will prove for $ r $.\\
	For $ \sigma_r+\sigma_{r-1}+\cdots+\sigma_{r-i}>0\,,\,0\leq i\leq r-2 $, it follows from the induction hypothesis that
	\begin{align*}
		&\sum_{n_1\leq x}\sum_{n_2=1}^{\infty}\cdots	\sum_{n_r=1}^{\infty}\frac{\chi_1(n_1)\chi_2(n_2)\cdots\chi_r(n_r)}{(n_0+n_1)^{s_1}(n_0+n_1+n_2)^{s_2}\cdots (n_0+n_1+n_2+\cdots+n_r)^{s_r}}\\
		&=\sum_{n_1\leq x}\frac{\chi_1(n_1)}{(n_0+n_1)^{s_1}}\int_{1}^\infty\int_{1}^\infty\cdots\int_{1}^\infty\prod_{i=2}^r\left(\sum_{n_i\leq t_i}\chi_i(n_i)\right)\\
		&\qquad\times\left(\sum_{\substack{k_2+\cdots+k_r=r-1\\k_2+\cdots+k_i\leq i-1\,,\,2\leq i\leq r\\0\leq k_i\leq i-1\,,\,2\leq i\leq r}}{1\choose k_2}\prod_{i=3}^{r}{i-1-k_2-\cdots-k_{i-1}\choose k_i}\right.\\
		&\qquad\qquad\left.\times\prod_{i=2}^r\frac{(s_i)_{k_i}}{\left(n_0+n_1+\sum_{j=2}^i t_j\right)^{s_i+k_i}}\right)dt_r\cdots dt_2\\
		&=\int_{1}^\infty\int_{1}^\infty\cdots\int_{1}^\infty\prod_{i=2}^r\left(\sum_{n_i\leq t_i}\chi_i(n_i)\right)\\
		&\times\Biggl(\sum_{\substack{k_2+\cdots+k_r=r-1\\k_2+\cdots+k_i\leq i-1\,,\,2\leq i\leq r\\0\leq k_i\leq i-1\,,\,2\leq i\leq r}}{1\choose k_2}\prod_{i=3}^{r}{i-1-k_2-\cdots-k_{i-1}\choose k_i}\Biggr.\\
		&\Biggl.\times\sum_{n_1\leq x}\frac{\chi_1(n_1)}{(n_0+n_1)^{s_1}}\prod_{i=2}^r\frac{(s_i)_{k_i}}{\left(n_0+n_1+\sum_{j=2}^i t_j\right)^{s_i+k_i}}\Biggr)dt_rdt_{r-1}\cdots dt_2\,.	
	\end{align*}
Now applying the Abel summation formula to the term 
$$\sum_{n_1\leq x}\frac{\chi_1(n_1)}{(n_0+n_1)^{s_1}}\prod_{i=2}^r\frac{(s_i)_{k_i}}{\left(n_0+n_1+\sum_{j=2}^i t_j\right)^{s_i+k_i}} 
$$ 
and then using Lemma \ref{lemma1} , we get that
\begin{align}
	\label{exp6}
	&\sum_{n_1\leq x}\sum_{n_2=1}^{\infty}\cdots	\sum_{n_r=1}^{\infty}\frac{\chi_1(n_1)\chi_2(n_2)\cdots\chi_r(n_r)}{n_1^{s_1}(n_1+n_2)^{s_2}\cdots (n_1+n_2+\cdots+n_r)^{s_r}}\\\nonumber
	&=\sum_{\substack{k_2+\cdots+k_r=r-1\\k_2+\cdots+k_i\leq i-1\,,\,2\leq i\leq r\\0\leq k_i\leq i-1\,,\,2\leq i\leq r}}{1\choose k_2}\prod_{i=3}^{r}{i-1-k_2-\cdots-k_{i-1}\choose k_i}\\\nonumber
	&\qquad\times\left(\sum_{n_1\leq x}\chi_1(n_1)\right)\frac{1}{(n_0+x)^{s_1}}\int_{1}^\infty\int_{1}^\infty\cdots\int_{1}^\infty\prod_{i=2}^r\left(\sum_{n_i\leq t_i}\chi_i(n_i)\right)\\\nonumber
	&\qquad\times\prod_{i=2}^r\frac{(s_i)_{k_i}}{\left(n_0+x+\sum_{j=2}^i t_j\right)^{s_i+k_i}}dt_r\cdots dt_2\\\nonumber
	&+\sum_{\substack{m_1+\cdots+m_r=r\\m_1+\cdots+m_i\leq i\,,\,1\leq i\leq r\\0\leq m_i\leq i\,,\,1\leq i\leq r}}{1\choose m_1}\prod_{i=2}^{r}{i-m_1-\cdots-m_{i-1}\choose m_i}\\\nonumber
	&\qquad\times\int_{1}^\infty\int_{1}^\infty\cdots\int_{1}^\infty\prod_{i=2}^r\left(\sum_{n_i\leq t_i}\chi_i(n_i)\right)\\\nonumber
	&\qquad\times\left(\int_1^x\sum_{n_1\leq t_1}\chi_1(n_1)\prod_{i=1}^r\frac{(s_i)_{m_i}}{\left(n_0+\sum_{j=1}^i t_j\right)^{s_i+m_i}}dt_1\right)dt_r\cdots dt_2\\\nonumber	
	&=\sum_{\substack{k_2+\cdots+k_r=r-1\\k_2+\cdots+k_i\leq i-1\,,\,2\leq i\leq r\\0\leq k_i\leq i-1\,,\,2\leq i\leq r}}{1\choose k_2}
	\prod_{i=3}^{r}{i-1-k_2-\cdots-k_{i-1}\choose k_i}U_{k_2,\ldots,k_r}(x)\\\nonumber
	&+\sum_{\substack{m_1+\cdots+m_r=r\\m_1+\cdots+m_i\leq i\,,\,1\leq i\leq r\\0\leq m_i\leq i\,,\,1\leq i\leq r}}{1\choose m_1}\prod_{i=2}^{r}{i-m_1-\cdots-m_{i-1}\choose m_i}
	V_{m_1,\ldots,m_r}(x),
\end{align}
say.

Next we will discuss that the limit of the above, when $ x $ tends to $ \infty $, exists.

For integers $ k_2,k_3,\ldots,k_r $ satisfying $ k_2+\cdots+k_r=r-1\,,\,k_2+\cdots+k_i\leq i-1\,,\,\\0\leq k_i\leq i-1 $, $ 2\leq i\leq r $ ensures that $ k_r+k_{r-1}\cdots+k_{r-i}\geq i+1\,,\,0\leq i\leq r-2 $ .\\
Then for $ \sigma_r+\sigma_{r-1}+\cdots+\sigma_{r-i}>0\,,\,0\leq i\leq r-1 $, we have
$$
\sigma_r+\cdots+\sigma_{r-i}+k_r+\cdots+k_{r-i}>i+1 \quad (0\leq i\leq r-2).
$$
Therefore we have

\begin{align*}
	&|U_{k_2,\ldots,k_r}(x)|\\
	&\leq\frac{\prod_{i=1}^rq_i}{(n_0+x)^{\sigma_1}}\int_{1}^\infty\cdots\int_{1}^\infty\prod_{i=2}^{r-1}\frac{\left|(s_i)_{k_i}\right|}{\left(n_0+x+\sum_{j=2}^it_j\right)^{\sigma_i+k_i}}\\
	&\qquad\times\left(\int_1^\infty\frac{\left|(s_r)_{k_r}\right|}{\left(n_0+x+\sum_{j=2}^rt_j\right)^{\sigma_r+k_r}}dt_r\right)dt_{r-1}\cdots dt_{2}\\
	&=\frac{\prod_{i=1}^rq_i}{(\sigma_r+k_r-1)(n_0+x)^{\sigma_1}}\int_{1}^\infty\cdots\int_{1}^\infty\prod_{i=2}^{r-1}\frac{\left|(s_i)_{k_i}\right|}{\left(n_0+x+\sum_{j=2}^it_j\right)^{\sigma_i+k_i}}\\
	&\qquad\times\frac{\left|(s_r)_{k_r}\right|}{\left(n_0+x+1+\sum_{j=2}^{r-1}t_j\right)^{\sigma_r+k_r-1}}dt_{r-1}\cdots dt_{2}\\
	&\leq \frac{\prod_{i=1}^r q_i\left|(s_r)_{k_r}\right|}{\sigma_r(n_0+x)^{\sigma_1}}\int_{1}^\infty\cdots\int_{1}^\infty\prod_{i=2}^{r-2}\frac{\left|(s_i)_{k_i}\right|}{\left(n_0+x+\sum_{j=2}^it_j\right)^{\sigma_i+k_i}}\\
	&\qquad\times\left(\int_1^\infty\frac{\left|(s_{r-1})_{k_{r-1}}\right|}{\left(n_0+x+\sum_{j=2}^{r-1}t_j\right)^{\sigma_r+\sigma_{r-1}+k_r+k_{r-1}-1}}dt_{r-1}\right)dt_{r-2}\cdots dt_{2}\,.	
\end{align*}
Continuing in the same manner, we arrive at
\begin{align*}
|U_{k_2,\ldots,k_r}(x)|
	\leq\left(\prod_{i=1}^r q_i\right)\left(\prod_{i=2}^r\frac{\left|(s_i)_{k_i}\right|}{\sigma_r+\sigma_{r-1}+\cdots+\sigma_i}\right)\frac{1}{(n_0+x)^{\sigma_r+\sigma_{r-1}+\cdots+\sigma_1}}\,.	
\end{align*}
This implies that the limit of $U_{k_2,\ldots,k_r}(x)$ when $ x $ tends to $ \infty $ 
exists and
\begin{align}
\label{exp7}
\lim_{x\to\infty}U_{k_2,\ldots,k_r}(x)=0
\end{align}
for $\sigma_r+\sigma_{r-1}+\cdots+\sigma_{r-i}>0\,,\,0\leq i\leq r-1.$

For integers $ m_1,m_2,\ldots,m_r $ satisfying $ m_1+m_2+\cdots+m_r=r\,,\,m_1+\cdots+m_i\leq i\,,\,\\0\leq m_i\leq i $, $ 1\leq i\leq r $ ensures that $ m_r+m_{r-1}\cdots+m_{r-i}\geq i+1\,,\,0\leq i\leq r-1 $.\\
Then for $ \sigma_r+\sigma_{r-1}+\cdots+\sigma_{r-i}>0\,,\,0\leq i\leq r-1 $, we have
\begin{align*}
&|V_{m_1,\ldots,m_r}(x)|\\
	&\leq\prod_{i=1}^r q_i\int_{1}^\infty\int_{1}^\infty\cdots\int_{1}^\infty\int_1^x\prod_{i=1}^r\frac{\left|(s_i)_{m_i}\right|}{\left(n_0+\sum_{j=1}^i t_j\right)^{\sigma_i+m_i}}dt_1dt_r\cdots dt_2\\
	&=\prod_{i=1}^r q_i\int_{1}^x\int_{1}^\infty\cdots\int_{1}^\infty\prod_{i=1}^{r-1}\frac{\left|(s_i)_{m_i}\right|}{\left(n_0+\sum_{j=1}^i t_j\right)^{\sigma_i+m_i}}\\
	&\qquad\times\left(\int_1^\infty\frac{\left|(s_r)_{m_r}\right|}{\left(n_0+\sum_{j=1}^r t_j\right)^{\sigma_r+m_r}} dt_r\right)dt_{r-1}\cdots dt_{1}\\
	&\leq \frac{\prod_{i=1}^r q_i\left|(s_r)_{m_r}\right|}{\sigma_r}\int_{1}^x\int_{1}^\infty\cdots\int_{1}^\infty\prod_{i=1}^{r-1}\frac{\left|(s_i)_{m_i}\right|}{\left(n_0+\sum_{j=1}^i t_j\right)^{\sigma_i+m_i}}\\
	&\qquad\times\frac{1}{\left(n_0+\sum_{j=1}^{r-1} t_j+1\right)^{\sigma_r+m_r-1}}dt_{r-1}\cdots dt_{1}\\
	&\leq \frac{\prod_{i=1}^r q_i\left|(s_r)_{m_r}\right|}{\sigma_r}\int_{1}^x\int_{1}^\infty\cdots\int_{1}^\infty\prod_{i=1}^{r-2}\frac{\left|(s_i)_{m_i}\right|}{\left(n_0+\sum_{j=1}^i t_j\right)^{\sigma_i+m_i}}\\
	&\qquad\times\left(\int_1^\infty\frac{\left|(s_{r-1})_{m_{r-1}}\right|}{\left(n_0+\sum_{j=1}^{r-1} t_j\right)^{\sigma_r+\sigma_{r-1}+m_r+m_{r-1}-1}} dt_{r-1}\right)dt_{r-2}\cdots dt_{1}\,.
\end{align*}
Continuing in the same manner, we arrive at
\begin{align*}
&|V_{m_1,\ldots,m_r}(x)|\\
	&\leq\left(\prod_{i=1}^r q_i\right)\left(\prod_{i=1}^r\frac{(\left|s_i\right|)_{m_i}}{\sigma_r+\sigma_{r-1}+\cdots+\sigma_i}\right)\left(\frac{1}{(n_0+1)^{\sigma_r+\sigma_{r-1}+\cdots+\sigma_1}}-\frac{1}{(n_0+x)^{\sigma_r+\sigma_{r-1}+\cdots+\sigma_1}}\right)\,. 
\end{align*}
This implies that the limit
\begin{equation}
	\label{exp8}
	\lim_{x\to\infty}V_{m_1,\ldots,m_r}(x)
\end{equation}
exists for $\sigma_r+\sigma_{r-1}+\cdots+\sigma_{r-i}>0\,,\,0\leq i\leq r-1\,.$
 \vspace{2mm}\\
Therefore by letting $ x $ tend to $ \infty $ in \eqref{exp6} and noting \eqref{exp7} and \eqref{exp8}, it follows that
\begin{align*}
	&\sum_{n_1=1}^\infty\sum_{n_2=1}^{\infty}\cdots	\sum_{n_r=1}^{\infty}\frac{\chi_1(n_1)\chi_2(n_2)\cdots\chi_r(n_r)}{n_1^{s_1}(n_1+n_2)^{s_2}\cdots (n_1+n_2+\cdots+n_r)^{s_r}}\\
	&=\sum_{\substack{m_1+\cdots+m_r=r\\m_1+\cdots+m_i\leq i\,,\,1\leq i\leq r\\0\leq m_i\leq i\,,\,1\leq i\leq r}}{1\choose m_1}\prod_{i=2}^{r}{i-m_1-\cdots-m_{i-1}\choose m_i}\\
	&\int_{1}^\infty\int_{1}^\infty\cdots\int_{1}^\infty\prod_{i=2}^r\left(\sum_{n_i\leq t_i}\chi_i(n_i)\right)\left(\int_1^\infty\sum_{n_1\leq t_1}\chi_1(n_1)\prod_{i=1}^r\frac{(s_i)_{m_i}}{\left(\sum_{j=1}^i t_j\right)^{s_i+m_i}}dt_1\right)dt_r\cdots dt_2\\
	&=\int_{1}^\infty\int_{1}^\infty\cdots\int_{1}^\infty\prod_{i=1}^r\left(\sum_{n_i\leq t_i}\chi_i(n_i)\right)\\
	&\left(\sum_{\substack{m_1+\cdots+m_r=r\\m_1+\cdots+m_i\leq i\,,\,1\leq i\leq r\\0\leq m_i\leq i\,,\,1\leq i\leq r}}{1\choose m_1}\prod_{i=2}^{r}{i-m_1-\cdots-m_{i-1}\choose m_i}\prod_{i=1}^r\frac{(s_i)_{m_i}}{\left(\sum_{j=1}^it_j\right)^{s_i+m_i}}\right)dt_rdt_{r-1}\cdots dt_1
\end{align*} for $ \sigma_r+\sigma_{r-1}+\cdots+\sigma_{r-i}>0\,,\,0\leq i\leq r-1\,. $\\
This completes the proof of the theorems.
\end{proof}

	\end{document}